\documentclass[11pt]{amsart}
\usepackage[all]{xy}
\usepackage{amsfonts,amsmath,amssymb,amscd}
\usepackage{color}
\usepackage[pagebackref,colorlinks, linkcolor=blue, citecolor=red]{hyperref}
\usepackage{caption} 
\usepackage{hyperref}
\usepackage{mathtools}

\usepackage[top=1in, bottom=1in, left=1in, right=1in]{geometry}
\usepackage{graphicx}

\numberwithin{equation}{section}
\newtheorem{theorem}{Theorem}[section]
\newtheorem{proposition}[theorem]{Proposition}

\newtheorem{lemma}[theorem]{Lemma}
\newtheorem{coro}[theorem]{Corollary}
\newtheorem{definition}[theorem]{Definition}

\newtheorem{example}[theorem]{Example}

\title[Derivation module and Hilbert-Kunz multiplicity]{Derivation module and the 
Hilbert-Kunz multiplicity of the co-ordinate ring of a projective monomial curve}

\author{Om Prakash Bhardwaj}
\address{IIT Gandhinagar, Palaj, Gandhinagar, Gujarat-382355 India}
\email{om.prakash@iitgn.ac.in}
\thanks{}

\author{Indranath Sengupta}
\address{IIT Gandhinagar, Palaj, Gandhinagar, Gujarat-382355 India}
\email{indranathsg@iitgn.ac.in} 
\thanks{}


\begin{document}
\begin{abstract}
Let $n_0, n_1, \ldots, n_p$ be a sequence of positive integers such that $n_0 < n_1 < \cdots < n_p$ and $\mathrm{gcd}(n_0,n_1, \ldots,n_p) = 1$. Let $S =  \langle (0,n_p), (n_0,n_p-n_0),\ldots,(n_{p-1},n_p-n_{p-1}), (n_p,0) \rangle$ be an affine semigroup in $\mathbb{N}^2$.  The semigroup ring $k[S]$ is the co-ordinate ring of the  projective monomial curve in the projective space $\mathbb{P}_k^{p+1}$, which is defined parametrically by
\begin{center}
$x_0 = v^{n_p}, \quad x_1 = u^{n_0}v^{n_p-n_0},\quad  \ldots , \quad x_p= u^{n_{p-1}}v^{n_p-n_{p-1}}, \quad x_{p+1} = u^{n_p}$. 
\end{center}
In this article, we consider the case when $n_0, n_1, \ldots, n_p$ forms an arithmetic sequence, and give an explicit set of minimal generators for the derivation module $\mathrm{Der}_k(k[S])$. Further, we give an explicit formula for the Hilbert-Kunz multiplicity of the co-ordinate ring of a projective monomial curve.
\end{abstract}
\maketitle

\section{Introduction}
Let $k$ be a field of characteristic 0, and $(R,\mathfrak{m})$ be a local (graded) $k$-algebra. Finding an explicit set of minimal generators for the derivation module $\mathrm{Der}_k(R)$ of $(R,\mathfrak{m})$ is an important problem in the literature, where $\mathrm{Der}_k(R)$ denote the $R$-module of $k$-derivations of $R$. Previously, this problem has been studied by many authors, for reference see ( \cite{kraft}, \cite{patil-sengupta}, \cite{patil-singh}, \cite{seidenberg}, \cite{jonathan}).

Let $r \geq 1$, and $S$ be an affine semigroup in $\mathbb{N}^r$ generated by $a_0,a_1, \ldots, a_p$. The semigroup ring $k[S]:=\oplus_{s \in S} k {\bf t}^s $ of $S$ is a $k$-subalgebra of the polynomial ring $k[t_1,\ldots,t_r]$, where $t_1,\ldots,t_r$ are indeterminates and ${\bf t}^s = \prod_{i=1}^r t_i^{s_i}$ for all $s = (s_1,\ldots,s_r) \in S$. If $r =1$, then $S$ is a submonoid in $\mathbb{N}$, and the semigroup ring $k[S]$ is isomorphic to a numerical semigroup ring. When $S$ is a numerical semigroup, Kraft in \cite{kraft}, proved that the derivation module $\mathrm{Der}_k(k[S])$ is minimally generated by the set $\{t^{\alpha+1} \frac{\partial}{\partial t} \mid \alpha \in \mathrm{PF}(S) \cup \{0\}\}$, where $\mathrm{PF}(S)$ denotes the set of pseudo-Frobenius elements of $S$. For $r \geq 2$, if $S$ is an affine semigroup in $\mathbb{N}^r$, then Tamone and Molinelli (\cite{tamone}, \cite{tamone2}), give the structure of $k$-derivations of $k[S]$ for some special type of semigroups. In \cite{tamone}, they consider the affine semigroup $S =  \langle (0,n_p), (n_0,n_p-n_0),\ldots,(n_{p-1},n_p-n_{p-1}), (n_p,0) \rangle$ in $\mathbb{N}^2$, where $n_0, n_1, \ldots, n_p$ is a sequence of positive integers such that $n_0 < n_1 < \cdots < n_p$ and $\mathrm{gcd}(n_0,n_1, \ldots ,n_p) = 1$. For $i = 1,2$, let $S_1$ and $S_2$ be the natural projections onto first and second component of $S$. With the assumption on the generators, note that $S_1$ and $S_2$ are numerical semigroups. For these type of affine semigroups, when $k[S]$ is Cohen-Macaulay, they give the structure of the derivations of the derivation module $\mathrm{Der}_k(k[S])$ using the set of pseudo-Frobenius elements of $S_1$ and $S_2$ (see \ref{genderk}). In this article, we give the explicit generators of the derivation module of the co-ordinate ring of the projective monomial curve defined by an arithmetic sequence using the structures of derivations given by Tamone and Molinelli \cite{tamone}.

Now, we summarize the contents of the paper. In this article, we consider the affine semigroup $S =  \langle (0,n_p), (n_0,n_p-n_0),\ldots,(n_{p-1},n_p-n_{p-1}), (n_p,0) \rangle$ in $\mathbb{N}^2$, where $n_0, n_1, \ldots, n_p$ is a sequence of positive integers such that $n_0 < n_1 < \cdots < n_p$ and $\mathrm{gcd}(n_0,n_1, \ldots ,n_p) = 1$, and $k[S]$ the semigroup algebra associated to $S$, which is isomorphic to the co-ordinate ring of a projective monomial curve in $\mathbb{P}_k^{p+1}$.

In section 2, we recall some definitions about numerical semigroups and when $k[S]$ is Cohen-Macaulay, we summarize the structure of the generating set of the derivation module of $k[S]$ in Theorem \ref{genderk}. In section 3, we consider $p=1$, i.e the sequence of positive integers $n_0, n_1$ such that $n_0 < n_1$ and $\mathrm{gcd}(n_0,n_1)=1$. In Theorem \ref{2-generated}, we give the explicit set of minimal generators of the derivation module for the co-ordinate ring of the projective monomial curve in $\mathbb{P}_k^2$ defined by the positive integers $n_0$ and $n_1$.

In section 4, we consider and the sequence $n_0,n_1,n_2, \ldots,n_p$, which forms a minimal arithmetic sequence and consider the affine semigroup $S = \langle (0,n_p), (n_0,n_p-n_0), (n_1,n_p-n_1),\ldots, (n_{p-1},n_p-n_{p-1}), (n_p,0) \rangle $. From \cite{bermejo-isabel}, we know that $k[S]$ is Cohen-Macaulay. In Proposition \ref{mu=type+3}, we prove that $\mu(\mathrm{Der}_k(k[S]_{\mathfrak{m}}))= r+3$, where $r$ is the Cohen-Macaulay type of $k[S]$ and $\mathfrak{m}$ is the maximal homogeneous ideal of $k[S]$. In Corollary \ref{mu}, for $n_0 = ap+b$, $0 \leq b < p$, we write the formula for $\mu(\mathrm{Der}_k(k[S]_{\mathfrak{m}}))$, which is,
\[\mu(\mathrm{Der}_k(k[S]_{\mathfrak{m}}))  = \begin{cases} 
 4 & \text{if} \quad p=1;\\
 p+2 & \text{if} \quad p \geq 2, b =0; \\
 p+3 & \text{if} \quad p \geq 2, b =1; \\
 b+2 &  \text{if} \quad p \geq 2, 1 < b < p.
       \end{cases}
    \]
In Theorem \ref{Der(p+1)}, we give an explicit set of minimal generators for $\mathrm{Der}_k (k[S])$. In section 5, we compute the Hilbert-Kunz multiplicity of the semigroup algebra $k[S]$. In Theorem \ref{Hilbert-KunzMulti}, we prove that the Hilbert-Kunz multiplicity of $k[S]$ is equal to  $1 + \frac{1}{n_p}\left( \sum_{r=1}^p (n_r-1)(n_r-n_{r-1}) \right)$, where $n_0 = 0$. It is interesting to note that in the case of semigroup algebras, the computation of Hilbert-Kunz multiplicity is independent of the characteristic of the base field.

\section{Preliminaries}
Throughout the article, $\mathbb{Z}$ and $\mathbb{N}$ denote the sets of integers and non-negative integers respectively.
\begin{definition}
Let $S$ be a submonoid of $\mathbb{N}$ such that $\mathbb{N} \setminus S$ is finite, then $S$ is called a numerical semigroup. Equivalently, there exist $n_0,\ldots,n_p \in \mathbb{N}$ such that $\mathrm{gcd}(n_0,\ldots,n_p) = 1$ and 
\[
S =  \langle n_0, \ldots, n_p \rangle = \left\lbrace \sum_{i=0}^{p} \lambda_i n_i \mid \lambda_i \in \mathbb{N}, \forall i \right\rbrace.
\]
\end{definition}

Since $\mathbb{N}\setminus S$ is finite, the largest number in $\mathbb{N}\setminus S$ is called the Frobenius number of $S$, and it is denoted by $F(S)$. 

\begin{definition}
Let $S$ be a numerical semigroup. For any $s \in S$, if $s = \sum_{i=0}^p \lambda_in_i$ is the unique expression for $s$ in $S$, then we say $s$ has unique factorization in $S$. In other words, we say $s$ has a unique factorization in $S$ if given any two expressions of $s$, $s = \sum_{i=0}^p \lambda_in_i$ and $s = \sum_{i=0}^p \lambda_i'n_i$, we have $\lambda_i = \lambda_i'$ for all $i \in [0,p]$.
\end{definition}

Given $0 \neq s \in S$, the set of lengths of $s$ in $S$ is defined as
\[
\mathcal{L}(s) = \left\lbrace \sum_{i=0}^{p} \lambda_i  \mid s = \sum_{i=0}^p\lambda_in_i, \quad \lambda_i \in \mathbb{N} \right\rbrace.
\]

\begin{definition}
A subset $T \subseteq  S$ is called homogeneous if either it is empty or $\mathcal{L}(s)$ is singleton for all $0 \neq s \in T$.
\end{definition}

\begin{definition}
Let $S$ be a numerical semigroup and $a$ be a non-zero element of $S$. The
set $\mathrm{Ap}(S, a) = \{s \in S \mid s-a \notin S \}$ is called the Ap\'{e}ry set of $S$ with respect to $a$.
\end{definition}

\begin{definition}
Let $S$ be a numerical semigroup. An element $f \in \mathbb{Z} \setminus S$ is called a pseudo-Frobenius number if $f + s \in S$ for all $s \in S\setminus \{0\}$. The set of pseudo-Frobenius numbers of $S$ is denoted by $\mathrm{PF}(S)$. Note that $F(S) \in \mathrm{PF}(S)$ and $F(S)$ is the maximum element of $\mathrm{PF}(S)$.
\end{definition}

The cardinality of the set of pseudo-Frobenius elements is known as the type of the numerical semigroup $S$, which is equal to the Cohen-Macaulay type of the numerical semigroup ring $k[S]$. Let $a_0,a_1, \ldots , a_p \in \mathbb{N}^r$ then
\[S = \langle a_0,a_1, \ldots ,a_p \rangle = \left\lbrace \sum_{i=0}^{p} \lambda_i a_i \mid \lambda_i \in \mathbb{N}, \forall i \right\rbrace\]
is called an \textit{affine semigroup} generated by  $a_0,a_1, \ldots,  a_p$. For $r=1$, affine semigroups correspond to numerical semigroups. Let $k$ be a field, the semigroup ring $k[S]:=\oplus_{s \in S} k {\bf t}^s $ of $S$ is a $k$-subalgebra of the polynomial ring $k[t_1,\ldots,t_r]$, where $t_1,\ldots,t_r$ are indeterminates and ${\bf t}^s = \prod_{i=1}^r t_i^{s_i}$, for all $s = (s_1,\ldots,s_r) \in S$. The semigroup ring $k[S] = k[{\bf t}^{a_0},{\bf t}^{a_1}, \ldots, {\bf t}^{a_p}]$ of $S$ can be represented as a quotient of a polynomial ring using a canonical surjection $\pi : k[x_0,x_1,\ldots,x_p] \rightarrow k[S]$, given by $\pi(x_i) = {\bf t}^{a_i}$ for all $i=0,1,\ldots,p.$ 
\medskip

Let $n_0,n_1, \ldots, n_p$ be a sequence of positive integers such that $n_0 < n_1 < \cdots < n_p$. Let $\mathcal{C}$ be a projective monomial curve in the projective space $\mathbb{P}_K^{p+1}$, defined parametrically by
\begin{center}
$x_0 = u^{n_p}; \quad x_1 = t^{n_0}u^{n_p-n_0}; \quad  \ldots \quad ; \quad  x_p = t^{n_{p-1}}u^{n_p - n_{p-1}}; \quad x_{p+1} = t^{n_p}$.
\end{center}

Let $k[\mathcal{C}]$ denote the co-ordinate ring of $\mathcal{C}$. Then 
$k[\mathcal{C}] = k[S]$, where $S = \langle (0,n_p), (n_0,n_p-n_0), (n_1,n_p-n_1),\ldots, (n_{p-1},n_p-n_{p-1}), (n_p,0) \rangle$ is an affine semigroup in $\mathbb{N}^2$. For such affine semigroup rings, we recall 
the following theorem from \cite{tamone}, which gives a set of generators of the derivation module $\mathrm{Der}_k(k[S])$.

\begin{theorem}\cite{tamone}\label{genderk}
Let $S = \langle (0,n_p), (n_0,n_p-n_0), (n_1,n_p-n_1),\ldots, (n_{p-1},n_p-n_{p-1}), (n_p,0) \rangle$ be an affine semigroup in $\mathbb{N}^2$, where $n_0 < n_1 < \cdots < n_p$ and $\mathrm{gcd}(n_0,n_1, \ldots,n_p) = 1$. Let $S_1$ and $S_2$ be the numerical semigroups corresponding the natural projections to the first and second components of $S$ respectively. If the semigroup ring $k[S]$ is Cohen-Macaulay then the derivation module $\mathrm{Der}_k(k[S])$ is generated by $D_1 \cup \{u \frac{\partial}{\partial u}\} \cup D_2 \cup \{ t \frac{\partial}{\partial t} \}$, where $D_{1}$ and $D_{2}$ are defined below. 
\begin{enumerate}
\item If $S_2 \neq \mathbb{N}$, then $D_1 = \{t^{\beta}u^{\alpha +1}\frac{\partial}{\partial u} \mid \alpha \in \mathrm{PF}(S_2)\}$, and $\beta$ is the least positive integer such that the pair $(\beta,\alpha)$ satisfy
\[
(\beta,\alpha)+(n,n_p-n) \in S \quad \text{for each} \quad
n \in \{0,n_0,\ldots,n_{p-1}\}.
\]

\item If $S_2 = \mathbb{N}$, then $D_1 = \{t^{1+cn_p}\frac{\partial}{\partial u}\}$, and $c$ is the least non-negative integer such that the pair $(1+cn_p,-1)$ satisfies
\[
(1+cn_p,-1)+(n,n_p-n) \in S \quad \text{for each} \quad
n \in \{0,n_0,\ldots,n_{p-1}\}.
\]

\item If $S_1 \neq \mathbb{N}$, then $D_2 = \{t^{\delta+1}u^{\gamma}\frac{\partial}{\partial t} \mid \delta \in \mathrm{PF}(S_1)\}$, and $\gamma$ is the least positive integer such that the pair $(\delta,\gamma)$ satisfy
\[
(\delta,\gamma)+(n,n_p-n) \in S \quad \text{for each} \quad
n \in \{n_0,\ldots,n_p\}.
\]

\item If $S_1 = \mathbb{N}$, then $D_2 = \{u^{1+en_p}\frac{\partial}{\partial t}\}$, and $e$ is the least non-negative integer such that the pair $(-1,1+en_p)$ satisfies
\[
(-1,1+en_p)+(n,n_p-n) \in S \quad \text{for each} \quad
n \in \{n_0,\ldots,n_p\}.
\]
\end{enumerate}
\end{theorem}

\section{Derivations in $\mathbb{P}_k^2$}

In this section, we give the explicit set of minimal generators of the derivation module for the co-ordinate ring of a projective monomial curve defined by the positive integers 
$n_0$ and $n_1$, such that $n_0 < n_1$ and $\mathrm{gcd}(n_0,n_1) = 1$.

\begin{proposition}\label{2-generated}
Let $S = \langle (0,n_1), (n_0,n_1-n_0), (n_1,0) \rangle$ be an affine semigroup 
in $\mathbb{N}^2$, such that $ n_0 < n_1$ and $\mathrm{gcd}(n_0,n_1) = 1$. Then we have the following:
\begin{enumerate}
\item If $S_1, S_2 \neq \mathbb{N}$, then the derivation module $\mathrm{Der}_k(k[S])$ is mimimally generated by
\[
\left\lbrace t \frac{\partial}{\partial t}, t^{n_0(n_1-1)-n_1+1}u^{(n_1-1)(n_1-n_0)}\frac{\partial}{\partial t}, u \frac{\partial}{\partial u}, t^{n_0(n_1-1)}u^{(n_1-1)(n_1-n_0)-n_1+1}\frac{\partial}{\partial u} \right\rbrace .
\]

\item If $S_1 = \mathbb{N}$ and $S_2 \neq \mathbb{N}$, then the derivation module $\mathrm{Der}_k(k[S])$ is mimimally generated by
\[
\left\lbrace t \frac{\partial}{\partial t}, u^{1+(n_1-2)n_1}\frac{\partial}{\partial t}, u \frac{\partial}{\partial u}, t^{n_1-1}u^{(n_1-1)(n_1-2)}\frac{\partial}{\partial u} \right\rbrace .
\]

\item If $S_1 \neq \mathbb{N}$ and $S_2 = \mathbb{N}$, then the derivation module $\mathrm{Der}_k(k[S])$ is mimimally generated by
\[
\left\lbrace t \frac{\partial}{\partial t}, t^{n_0(n_1-1)-n_1+1}u^{n_1-1}\frac{\partial}{\partial t}, u \frac{\partial}{\partial u}, t^{n_0(n_1-1)}\frac{\partial}{\partial u} \right\rbrace .
\]

\item If $S_1 = S_2 = \mathbb{N}$, then the derivation module $\mathrm{Der}_k(k[S])$ is mimimally generated by
\[
\left\lbrace t \frac{\partial}{\partial t}, u\frac{\partial}{\partial t}, u \frac{\partial}{\partial u}, t\frac{\partial}{\partial u} \right\rbrace .
\]
\end{enumerate}
\end{proposition}

\begin{proof}
Let $S_1$ and $S_2$ be the projections of $S$ to the first and the 
second component of $S$, then we have $S_1 = \langle n_0,n_1 \rangle$ 
and $S_2 = \langle n_1-n_0, n_1 \rangle$. We will prove each case separately by using Theorem \ref{genderk}.
\medskip

\noindent\textbf{Case 1.} Suppose $S_1 , S_2 \neq \mathbb{N}$. From \cite[Proposition 2.13]{numerical}, we have $\mathrm{PF}(S_1) = \{n_0(n_1-1)-n_1\}$ and $\mathrm{PF}(S_2) = \{(n_1-1)(n_1-n_0)-n_1\}$. Let $\beta \in S_1$ be such that $(\beta, (n_1-1)(n_1-n_0)-n_1)+(0,n_1) = (\beta, (n_1-1)(n_1-n_0)) \in S$. Note that $(n_1-1)(n_1-n_0)$ has only one factorization in $S_2$. Therefore, 
the possible factorization of $(\beta, (n_1-1)(n_1-n_0))$ in $S$ will be
\[
(\beta, (n_1-1)(n_1-n_0)) = (n_1-1)(n_0,n_1-n_0) + \lambda (n_1,0), \quad \text{for some} \quad \lambda \geq 0.
\]
Therefore, we have $\beta = (n_1-1)n_0 + \lambda n_1 \geq n_0(n_1-1)$. Now for $\beta = n_0(n_1-1)$, we have 
\begin{align*}
(n_0(n_1-1),(n_1-1)(n_1-n_0)-n_1) + (0,n_1) & = (n_0(n_1-1),(n_1-1)(n_1-n_0))\\
&=(n_1-1)(n_0,n_1-n_0) 
\end{align*}
and
\begin{align*}
(n_0(n_1-1),(n_1-1)(n_1-n_0)-n_1) + (n_0,n_1-n_0) & = (n_0n_1, n_1(n_1-n_0-1))\\
&= n_0(n_1,0)+(n_1-n_0-1)(0,n_1).
\end{align*}
Now, let $\gamma \in S_2$ be such that $(n_0(n_1-1)-n_1,\gamma) + (n_1,0) = (n_0(n_1-1), \gamma) \in S$. Note that $n_0(n_1-1)$ has only one factorization in $S_1$. Therefore, the 
only possible factorization of $(n_0(n_1-1), \gamma)$ in $S$ is
\[
(n_0(n_1-1), \gamma) = (n_1-1)(n_0,n_1-n_0) + \lambda(0,n_1) \quad \text{for \, some} \quad \lambda \geq 0.
\]
Hence, we have $\gamma = (n_1-1)(n_1-n_0) + \lambda n_1 \geq (n_1-1)(n_1-n_0)$. Now for $\gamma = (n_1-1)(n_1-n_0)$, we have 
\begin{align*}
(n_0(n_1-1)-n_1,(n_1-1)(n_1-n_0)) + (n_0,n_1-n_0) & = (n_0n_1-n_1, n_1(n_1-n_0))\\
&= (n_0-1)(n_1,0)+(n_1-n_0)(0,n_1);
\end{align*}
and
\begin{align*}
(n_0(n_1-1)-n_1,(n_1-1)(n_1-n_0)) + (n_1,0) & = (n_0(n_1-1),(n_1-1)(n_1-n_0))\\
&= (n_1-1)(n_0,n_1-n_0).
\end{align*}

\noindent \textbf{Case 2.} Suppose $S_1 = \mathbb{N}$. Therefore, we must have $n_0 = 1$. 
Let $e$ be a non-negative integer such that $(-1,1+en_1) + (n_1,0) = (n_1-1,1+en_1) \in S$. Observe that $n_1-1$ has only one factorization in $S_1$. Therefore the only possible 
factorization of $(n_1-1,1+en_1)$ in $S$ is 
\[
(n_1-1,1+en_1) = (n_1-1)(1,n_1-1)+ \lambda(0,n_1), \quad \text{for \, some} \quad \lambda \geq 0. 
\]
Therefore, we have $1+en_1 \geq (n_1-1)^2$, which 
implies that $e \geq n_1-2$. Now for $e = n_1-2$, we have
\begin{align*}
(-1,1+(n_1-2)n_1)+(n_1,0) = (n_1-1,(n_1-1)^2) = (n_1-1)(1,n_1-1);
\end{align*}
and
\begin{align*}
(-1,1+(n_1-2)n_1)+(1,n_1-1) = (0,n_1(n_1-1)) = (n_1-1)(0,n_1).
\end{align*}

\noindent\textbf{Case 3.} Suppose $S_2=\mathbb{N}$. Therefore we must have $n_1-n_0 = 1$. 
Let $c$ be a non-negative integer such that $(1+cn_1,-1) + (0,n_1) = (1+cn_1,n_1-1) \in S$. Observe that $n_1-1$ has only one factorization in $S_2$.  Therefore the only 
possible factorization of $(1+cn_1,n_1-1)$ in $S$ is 
\[
(1+cn_1,n_1-1) = (n_1-1)(n_0,1)+ \lambda(n_1,0), \quad \text{for \, some} \quad \lambda \geq 0. 
\]
Therefore, we have $1+cn_1 \geq (n_1-1)n_0 = n_1(n_0-1)+1$. This implies that 
$c \geq n_0-1$. Now for $c = n_0-1$, we have
\begin{align*}
(1+(n_0-1)n_1,-1)+(0,n_1) = (1+(n_0-1)(n_0+1),(n_1-1)) = (n_1-1)(n_0,1);
\end{align*}
and
\begin{align*}
(1+(n_0-1)n_1,-1)+(n_0,1) &= (n_0n_1,0) = n_0(n_1,0).
\end{align*}

\noindent\textbf{Case 4.} Suppose $S_1=\mathbb{N}=S_2$. In this case, the only possibility is $S = \langle (0,2), (1,1), (2,0) \rangle $. From the arguements of cases 2 and 3, it is easy to observe that the derivation module $\mathrm{Der}_k(k[S])$ is minimally generated by $\left\lbrace t \frac{\partial}{\partial t}, u\frac{\partial}{\partial t}, u \frac{\partial}{\partial u}, t\frac{\partial}{\partial u} \right\rbrace$.
\end{proof}

\begin{example}{\rm 
Let $S = \langle (0,3), (1,2), (3,0) \rangle$. Here $S_1 = \mathbb{N}$ and $S_2 \neq \mathbb{N}$. The derivation module $\mathrm{Der}_k(k[S])$ is minimally generated by $\left\lbrace t \frac{\partial}{\partial t}, u^4\frac{\partial}{\partial t}, u \frac{\partial}{\partial u}, t^2u^2\frac{\partial}{\partial u} \right\rbrace$.
}
\end{example}

\begin{example}{\rm 
Let $S = \langle (0,9), (5,4), (9,0) \rangle$. Here $S_1 \neq \mathbb{N} \neq S_2$. The derivation module $\mathrm{Der}_k(k[S])$ is minimally generated by $\left\lbrace t \frac{\partial}{\partial t}, t^{32}u^{32}\frac{\partial}{\partial t}, u \frac{\partial}{\partial u}, t^{40}u^{24}\frac{\partial}{\partial u} \right\rbrace$.
}
\end{example}

\section{Derivations in $\mathbb{P}_k^{p+1}$}
For a numerical semigroup $S$, it is well known that $\mu(\mathrm{Der}_k(k[S]_{\mathfrak{m}}))= r+1$, where $r$ is the Cohen-Macaulay type of $k[S]$ and $\mathfrak{m}$ is the maximal homogeneous ideal of $k[S]$. Such a nice relation between $\mu(\mathrm{Der}_k(k[S]_{\mathfrak{m}}))$ and Cohen-Macaulay type of $k[S]$ does not hold in general for the affine semigroups (see \cite[Remark 3.6]{tamone}).

We now assume that $n_0,n_1,\ldots,n_p$ is an arithmetic sequence of positive integers i.e., for a fixed positive integer $d$, $n_i = n_0 + id$ for $i\in[0,p]$, such that $\mathrm{gcd}(n_0,n_1,\ldots,n_p) = 1$. Also assume that the sequence 
$n_0,n_1,\ldots,n_p$ forms a minimal generating set for a numerical semigroup, and we say 
that $n_0,n_1,\ldots,n_p$ is a minimal arithmetic sequence. Now define 
\[
S = \langle (0,n_p), (n_0,n_p-n_0), (n_1,n_p-n_1),\ldots, (n_{p-1},n_p-n_{p-1}), (n_p,0) \rangle ,
\]
an affine semigroup in $\mathbb{N}^2$. We will denote the natural projections to the first and second components of $S$ by $S_1$ and $S_2$ respectively.  These notations will be followed throughout the section.

\medskip
 From the \cite[Corollary 3.2]{bermejo-isabel}, we know that $k[S]$ is Cohen-Macaulay. The following Proposition gives a nice relation between $\mu(\mathrm{Der}_k(k[S]_{\mathfrak{m}}))$ and Cohen-Macaulay type of $k[S]$. 

\begin{proposition}\label{mu=type+3}
Let $S = \langle (0,n_p), (n_0,n_p-n_0), (n_1,n_p-n_1),\ldots, (n_{p-1},n_p-n_{p-1}), (n_p,0) \rangle$ be an affine semigroup in $\mathbb{N}^2$, where $n_0,n_1,\ldots,n_p$ is a minimal arithmetic sequence of positive integers such that $\mathrm{gcd}(n_0,n_1,\ldots,n_p) = 1$. Then $\mu(\mathrm{Der}_k(k[S]_{\mathfrak{m}}))= r+3$, where $r$ is the Cohen-Macaulay type of $k[S]$ and $\mathfrak{m}$ is the maximal homogeneous ideal of $k[S]$.
\end{proposition}
\begin{proof}
Since $n_0 < n_2 < \ldots < n_p $ is an arithmetic sequence of positive integers, then for a fixed positive integer $d$, we have $n_i = n_0 + id$ for $i\in[0,p]$. Since $S_1$ and $S_2$ are the numerical semigroups corresponding the natural projections to the first and second components of $S$ respectively, we have $S_1 = \langle n_0,n_1,\ldots,n_p \rangle$ and $S_2 = \langle d, n_p\rangle$. For $i=1,2$, let $r_i$ be the Cohen-Macaulay type of $S_i$. By \cite[Corollary 4.2]{pranjal-sengupta2}, we have $r_1 = r$ and by \cite[Proposition 2.13]{numerical}, we get $r_2=1$. Therefore by \cite[Corollary 3.5]{tamone}, we have 
\[
\mu(\mathrm{Der}_k(k[S]_{\mathfrak{m}})) = r_1+r_2+2 = r+3. \qedhere
\]
\end{proof}

\begin{coro}\label{mu}
With the assumptions of Proposition \ref{mu=type+3} and $n_0 = ap+b$, $0 \leq b < p$, we have
 \[\mu(\mathrm{Der}_k(k[S]_{\mathfrak{m}}))  = \begin{cases} 
 4 & \text{if} \quad p=1;\\
 p+2 & \text{if} \quad p \geq 2, b =0; \\
 p+3 & \text{if} \quad p \geq 2, b =1; \\
 b+2 &  \text{if} \quad p \geq 2, 1 < b < p.
       \end{cases}
    \]
\end{coro}
\begin{proof}
Since the case $p=1$ reduces to the Theorem \ref{2-generated}, we have $\mu(\mathrm{Der}_k(k[S]_{\mathfrak{m}}))= 4$ if $p=1$. We now assume that $p \geq 2$. Therefore by \cite[Theorem 3.1]{numata2}, the Cohen-macaulay type of $S_1$ is
\[r_1 = \begin{cases} 
 p-1 & \text{if} \quad  b =0; \\
 p & \text{if} \quad  b =1; \\
 b-1 &  \text{if} \quad 1 < b < p.
       \end{cases}
    \]
 Now, the result follows from Proposition \ref{mu=type+3}.  
\end{proof}

\begin{lemma}\label{homogeneousAperyset}
Let $S_1$ be the natural projection to the first component of $S$. Then the set $\mathrm{Ap}(S_1,n_p)$ is a homogeneous subset of $S_1$.
\end{lemma}
\begin{proof}
Define the map $\phi : k[x_0,\ldots,x_p] \longrightarrow k[t]$ such that $x_i \rightarrow t^{n_i}$ for $0 \leq i \leq p$. Then, $k[S_1] \cong \frac{k[x_0,\ldots,x_p]}{\mathrm{ker}(\phi)}$. From \cite{patil2}, we have a minimal generating set (say $B$) of $\mathrm{Ker}(\phi)$ such that one term of each non-homogeneous element of $B$ is divisible by $x_p$. Hence, the result follows from \cite[Proposition 3.9]{jafarihomnumsemigroup}.
\end{proof}

\begin{lemma}\label{equalsum}
Let $S_1$ be the natural projection to the first component of $S$ and $s \in \mathrm{Ap}(S_1,n_p)$. If $s = \sum_{i=0}^p \lambda_i n_i = \sum_{i=0}^p \lambda_i'n_i$ has two expressions in $S$, then
\[
\sum_{i=0}^{p} \lambda_i (p-i)d = \sum_{i=0}^{p} \lambda_i' (p-i)d.
\]
\end{lemma}
\begin{proof}
Since
\[
s = \sum_{i=0}^p \lambda_i n_i = \sum_{i=0}^p \lambda_i'n_i,
\]
we get
\[
\sum_{i=0}^p \lambda_i n_0 + \sum_{i=0}^p \lambda_i (id) = \sum_{i=0}^p \lambda_i' n_0 + \sum_{i=0}^p \lambda_i' (id).
\]
By Lemma \ref{homogeneousAperyset}, we have $\sum_{i=0}^p \lambda_i  = \sum_{i=0}^p \lambda_i'$. Thus, we have 
\[
\sum_{i=0}^p \lambda_i (pd) - \sum_{i=0}^p \lambda_i (id) = \sum_{i=0}^p \lambda_i' (pd) - \sum_{i=0}^p \lambda_i' (id).
\]
Therefore, we get 
\[
\sum_{i=0}^{p} \lambda_i (p-i)d = \sum_{i=0}^{p} \lambda_i' (p-i)d.
\]

\end{proof}

\begin{theorem}\label{Der(p+1)}
Suppose $p \geq 2$. Let $S = \langle (0,n_p), (n_0,n_p-n_0), (n_1,n_p-n_1),\ldots, (n_{p-1},n_p-n_{p-1}), (n_p,0) \rangle$ be an affine semigroup in $\mathbb{N}^2$, where $n_0,n_1,\ldots,n_p$ is a minimal arithmetic sequence of positive integers, i.e., for $i \in [1,p]$, $n_i = n_0 + id$ for some positive integer $d$, and $\mathrm{gcd}(n_0,n_1,\ldots,n_p) = 1$. Write $n_0 = ap + b$, $0 \leq b < p$, then we have the following:
\begin{enumerate}
\item If $b = 0$, then the derivation module $\mathrm{Der}_k(k[S])$ is mimimally generated by 
\[ 
\left\lbrace  u \frac{\partial}{\partial u}, t^{an_p+d}u^{(d-1)(n_p-1)}\frac{\partial}{\partial u} , t \frac{\partial}{\partial t}, t^{an_p - n_{p-i}+1}u^{d(n_p - i)}\frac{\partial}{\partial t} \mid 1 \leq i  \leq p-1 \right\rbrace.
\]

\item If $b = 1$, then the derivation module $\mathrm{Der}_k(k[S])$ is mimimally generated by 
\[ 
\left\lbrace  u \frac{\partial}{\partial u}, t^{an_p+d}u^{(d-1)(n_p-1)}\frac{\partial}{\partial u} , t \frac{\partial}{\partial t}, t^{an_p - n_{p-i}+1}u^{d(n_p - i)}\frac{\partial}{\partial t} \mid 1 \leq i  \leq p \right\rbrace.
\]

\item If $b \neq 0,1$, then the derivation module $\mathrm{Der}_k(k[S])$ is mimimally generated by
\[
\left\lbrace  u \frac{\partial}{\partial u}, t^{(a+1)n_p+d}u^{(d-1)(n_p-1)}\frac{\partial}{\partial u} , t \frac{\partial}{\partial t}, t^{an_p + id +1}u^{d(n_p - i)}\frac{\partial}{\partial t} \mid 1 \leq i  \leq b-1 \right\rbrace.
\]
\end{enumerate}
\end{theorem}
\begin{proof}
Let $S_1$ and $S_2$ be the numerical semigroups corresponding the natural projections to the first and second components of $S$ respectively. Then we have $\mathrm{PF}(S_2) = \{(d-1)n_p - d\}$. Also by \cite[Theorem 3.1]{numata2}, we can write the following formulas for $\mathrm{PF}(S_1)$.

If $b=0$, then
\begin{align*}
\mathrm{PF}(S_1) & = \left\lbrace an_0 + \ell d - n_0 \mid (a-1)p + 1 \leq \ell \leq ap-1 \right\rbrace \\
&= \left\lbrace (a-1)n_p + id \mid 1 \leq i \leq p-1 \right\rbrace \\
&= \left\lbrace an_p - n_{p-i} \mid 1 \leq i \leq p-1 \right\rbrace .
\end{align*}

If $b = 1$, then
\begin{align*}
\mathrm{PF}(S_1) & = \left\lbrace an_0 + \ell d - n_0 \mid (a-1)p + 1 \leq \ell \leq ap \right\rbrace \\
&= \left\lbrace an_p - n_{p-i} \mid 1 \leq i \leq p \right\rbrace .
\end{align*}

If $b \neq 0,1$, then
\begin{align*}
\mathrm{PF}(S_1) & = \left\lbrace (a+1)n_0 + \ell d - n_0 \mid ap + 1 \leq \ell \leq ap + b -1  \right\rbrace \\
&= \left\lbrace an_p + id \mid 1 \leq i \leq b-1 \right\rbrace .
\end{align*}

Now set $\beta = \begin{cases} 
 an_p + d & \text{if} \quad b = 0,1\\
 (a+1)n_p + d & \text{if} \quad b \neq 0,1 
       \end{cases}$, \, and \, $\alpha = (d-1)n_p-d$.\\

\vspace{.1cm}
Also set, for $i \in I$,
\begin{center}
$\delta_i = \begin{cases} 
 an_p - n_{p-i} & \text{if} \quad b = 0,1\\
 an_p + id & \text{if} \quad b \neq 0,1 
       \end{cases}$, \, and \, $\gamma_i = d(n_p-i)$,
\end{center}
where
\begin{center}
$I = \begin{cases}
[1,p-1] & \text{if} \quad b = 0;\\
[1,p] & \text{if} \quad b = 1;\\
[1,b-1] & \text{if} \quad b \neq 0,1.
\end{cases}$
\end{center}
Since $k[S]$ is Cohen-Macaulay, to prove $(\beta, \alpha) + (n,n_p - n) \in S$, 
for each $n \in \{0,n_0, \ldots, n_{p-1}\}$ and $(\delta_i,\gamma_i) + (n,n_p - n) \in S$, for each $i \in I$, $n \in \{n_0,n_1, \ldots , n_p\}$, is equivalent to prove that $(\beta, \alpha) + (n,n_p - n) \in (S_1 \times S_2) \cap G(S)$, for each $n \in \{0,n_0, \ldots, n_{p-1}\}$ and $(\delta_i,\gamma_i) + (n,n_p - n) \in (S_1 \times S_2) \cap G(S)$, for each $i \in I$, $n \in \{n_0,n_1, \ldots , n_p\}$, where $G(S)$ is the group generated by $S$ in $\mathbb{Z}^2$. 

Now, observe that $\beta \in S_1$ and since $\alpha = F(S_2)$, we have $\alpha + n_p-n \in S_2$, 
for each $n \in \{0,n_0, \ldots, n_{p-1}\}$. Therefore, we have $(\beta, \alpha) + (n,n_p - n) \in S_1  \times S_2$, 
for each $n \in \{0,n_0, \ldots, n_{p-1}\}$. Also, if $b = 0,1$, we have 
\begin{align*}
\beta + \alpha + n + n_p - n = an_p + d +(d-1)n_p -d + n_p = (a+d)n_p,
\end{align*}
and if $b \neq 0,1$, we have
\begin{align*}
\beta + \alpha + n + n_p - n = (a+1)n_p + d +(d-1)n_p -d + n_p = (a+d+1)n_p.
\end{align*}
Therefore by \cite[Lemma 4.1]{cavaliereNiesi}, we have  $(\beta, \gamma) + (n,n_p-n) \in G(S)$, for each $n \in \{0,n_0,\ldots,n_{p-1}\}$.
Now, since $\delta_i \in \mathrm{PF}(S_1)$, we have $\delta_i + n \in S_1$, for all $i,n$. Also, we have 
\begin{align*}
\gamma_i + n_p - n = d(n_p-i) + n_p - n &= (d-1)n_p - d + n_{p - (i-1)} + n_p -n \\
& = F(S_2) + n_{p - (i-1)} + n_p -n.
\end{align*}
Therefore, we have $(\delta_i,\gamma_i) + (n,n_p - n) \in (S_1 \times S_2)$, for each 
$i \in I$, $n \in \{n_0,n_1, \ldots , n_p\}$. Also, if $b = 0,1$, we have
\begin{align*}
\delta_i + \gamma_i + n + n_p -n = an_p - n_{p-i} + d(n_p- i) + n_p = (a+d)n_p,
\end{align*}
and if $b \neq 0,1$, we have
\begin{align*}
\delta_i + \gamma_i + n + n_p -n = an_p + id + d(n_p - i) + n_p = (a+d+1)n_p.
\end{align*}
Therefore by \cite[Lemma 4.1]{cavaliereNiesi}, we have $(\delta_i,\gamma_i) + (n,n_p - n) \in G(S)$ for each $i \in I$, $n \in \{n_0,n_1, \ldots , n_p\}$. 

To complete the proof, it is sufficient to prove that $\beta$ and $\gamma_i$'s are least positive integers such that $(\beta, \alpha) + (n,n_p - n) \in S$, 
for each $n \in \{0,n_0, \ldots, n_{p-1}\}$ and $(\delta_i,\gamma_i) + (n,n_p - n) \in S$, for each $i \in I$, $n \in \{n_0,n_1, \ldots , n_p\}$. Suppose $b \neq 0,1$, we have 
$$\alpha + n_p = (d-1)n_p - d +n_p = dn_p - d = d(n_0+pd)-d = (a+d)pd + (b-1)d.$$
If there exist $\beta$ such that $(\beta,\alpha)+ (0,n_p) \in S$, then we get
\begin{align*}
\beta \geq (a+d)n_0 + n_0 + (p-b+1)d & \geq an_0 + dn_0 + n_p - (b-1)d\\
& \geq an_0 + dap + n_p + d \\
& \geq (a+1)n_p + d.
\end{align*}
Thus, if $b \neq 0,1$ then $\beta = (a+1)n_p + d$ is minimal satisfying the required properties. Now, suppose $b \in {0,1}$, then we have
\[
\alpha + n_p = (d-1)n_p - d +n_p = (a+d-1)pd + (p-1)d \quad \text{if} \quad b=0,
\]
and 
\[
\alpha + n_p = (d-1)n_p - d +n_p = (a+d)pd \quad \text{if} \quad b=1.
\]
If there exist $\beta$ such that $(\beta,\alpha)+ (0,n_p) \in S$, then we get
\[
\beta \geq (a+d-1)n_0 + n_0 + d  \geq an_0 + dn_0 + d  \geq a(n_0+pd) + d \quad \text{if} \quad b=0,
\]
and
\[
\beta \geq (a+d)n_0   \geq an_0 + d(ap+1)  \geq a(n_0+pd) + d \quad \text{if} \quad b=1.
\]
Thus, if $b \in \{0,1\}$ then $\beta = an_p + d$ is minimal satisfying the required properties.
Now, since we have $\delta_i \in \mathrm{PF}(S_1)$ for all $i\in I$, then $\delta_i + n_p \in \mathrm{Ap}(S_1,n_p)$ for all $i \in I$. Suppose $b\in \{0,1\}$, we have
\begin{align*}
\delta_i + n_p = an_p - n_{p-i} + n_p &= a(n_0+pd)-n_0-(p-i)d + n_0 + pd\\
&=(a+d-1)n_0 + apd + id + n_0 - n_0d\\
&= (a+d-1)n_0 + id +n_0 -bd\\
&= (a+d-1)n_0 + n_{i-b}.
\end{align*}
If there exist $\gamma_i$ such that $(\delta_i,\gamma_i)+ (n_p,0) \in S$, then by Lemma \ref{equalsum}, we get
\begin{align*}
\gamma_i \geq (a+d-1)pd + pd - (i-b)d & \geq (a+d)pd-id + bd \\
& \geq(a+d)pd - id + n_0d - apd\\
& \geq d(n_0 + pd) -id.
\end{align*}
Thus, if $b \in \{0,1\}$ then $\gamma_i = d(n_p-i)$ is minimal satisfying the required properties. Now, suppose $b \neq 0,1$. Therefore, we have 
\begin{align*}
\delta_i + n_p = an_p +id + n_p &= a(n_0+pd) +id + n_0 + pd\\
&=(a+d)n_0 + n_0 + id -(n_0-ap)d + pd \\
&= (a+d)n_0 + n_0 + (p+i-b)d\\
&= (a+d)n_0 + n_{p+i-b}.
\end{align*}
If there exist $\gamma_i$ such that $(\delta_i,\gamma_i)+ (n_p,0) \in S$, then by Lemma \ref{equalsum}, we get
\begin{align*}
\gamma_i \geq (a+d)pd + pd - (p+i-b)d & \geq (a+d)pd-id + bd \\
& \geq(a+d)pd - id + n_0d - apd\\
& \geq d(n_0 + pd) -id.
\end{align*}
Thus, if $b \neq 0,1$, then also $\gamma_i = d(n_p-i)$ is minimal satisfying the required properties. This completes the proof.
\end{proof}

\begin{example}{\rm 
Let $S = \langle (0,23), (11,12), (13,10), (15,8), (17,6), (19,4), (21,2), (23,0) \rangle$, then $S_1 = \langle 11, 13, 15, 17, 19, 21, 23 \rangle$ and $S_2 = \langle 2, 23 \rangle$. Here we have $n_0 = 11$ and $p = 6$, therefore we get $a = 1$, $b = 5$.  Therefore, we have $\mathrm{PF}(S_1) = \{25,27,29,31\}$ and $\mathrm{PF}(S_2) = \{21\}$. In the notation of proof of Theorem \ref{Der(p+1)}, we have  $\delta_1 = 25$, $\delta_2 = 27$, $\delta_3 = 29$, $\delta_4 = 31$. Note that

\[
\delta_1 + n_6 = 2 \cdot 11 + 2 \cdot 13 = 3 \cdot 11 + 15,
\]
\[
\delta_2 + n_6 = 11 + 3 \cdot 13 = 2 \cdot 11 + 13 + 15 = 3 \cdot 11 + 17,
\]
\[
\delta_3 + n_6 = 4 \cdot 13 = 11 + 2 \cdot 13 + 15 = 2 \cdot 11 + 2 \cdot 15 = 2 \cdot 11 + 13 + 17 = 3 \cdot 11 + 19,
\]
\[
\delta_4 + n_6 = 3 \cdot 13 + 15 = 11 + 13 + 2 \cdot 15 = 11 + 2 \cdot 13 + 17 = 2 \cdot 11 + 15 + 17 = 2 \cdot 11 + 13 + 19 = 3 \cdot 11 + 21,
\]

\medskip
\noindent
are the only factorizations of $\delta_1 + n_6,\delta_2 + n_6,\delta_3 + n_6,\delta_4 + n_6$ respectively. Also note that $\mathcal{L}(\delta_i + n_6) = 4$ for all $i \in [1,4]$. Therefore the minimal choices for $\gamma_i$'s such that $(\delta_i, \gamma_i) + (n_d,0) \in S$ are 

\[
\gamma_1 = 2 \cdot 12 + 2 \cdot 10 \quad \text{or} \quad 3 \cdot 12 + 8,
\]
\[
\gamma_2 = 12 + 3 \cdot 10 \quad \text{or} \quad 2 \cdot 12 + 10 + 8 \quad \text{or} \quad 3 \cdot 12 + 6,
\]
\[
\gamma_3 = 4 \cdot 10 \quad \text{or} \quad 12 + 2 \cdot 10 + 8 \quad \text{or} \quad 2 \cdot 12 + 2 \cdot 8 \quad \text{or} \quad 2 \cdot 12 + 10 + 6 \quad \text{or} \quad 3 \cdot 12 + 4,
\]
\[
\gamma_4 = 3 \cdot 10 + 8 \quad \text{or} \quad 12 + 10 + 2 \cdot 8 \quad \text{or} \quad 12 + 2 \cdot 10 + 6 \quad \text{or} \quad 2 \cdot 12 + 8 + 6 \quad \text{or} \quad 2 \cdot 12 + 10 + 4 \quad \text{or} \quad 3 \cdot 12 + 2.
\]

\medskip
In each case, we have $\gamma_1 = 44$, $\gamma_2 = 42$, $\gamma_3 = 40$ and $\gamma_4 = 38$. Further, we observe that these $\gamma_i$'s satisfy the condition $(\delta_i,\gamma_i) + (n,n_6 - n) \in S$, for each $i \in [1,4]$, $n \in \{n_0,n_1, \ldots , n_6\}$.
Now, since $\alpha = 21$, we have $\alpha+n_6 = 44$. Observe that $48$ is the smallest natural number such that $(48,\alpha)+(0,n_6) = (48,44) \in S$. Also, observe that $\beta = 48$ satisfies the property that $(\beta, \alpha) + (n, n_6-n) \in S$ for all $n \in \{n_0,n_1, \ldots , n_6\}$. Therefore, the set 
\[ 
\left\lbrace  u \frac{\partial}{\partial u}, t^{48}u^{22}\frac{\partial}{\partial u} , t \frac{\partial}{\partial t}, t^{26}u^{44}\frac{\partial}{\partial t},  t^{28}u^{42}\frac{\partial}{\partial t},  t^{30}u^{40}\frac{\partial}{\partial t},  t^{32}u^{38}\frac{\partial}{\partial t} \right\rbrace 
\]
forms a minimal generating set for $\mathrm{Der}_k(k[S])$.
}
\end{example}

\section{Hilbert-Kunz multiplicity}

Let $R$ be a $d$-dimensional graded $k$-algebra, with homogeneous maximal ideal $\mathfrak{m}$. Let $M$ be a finite $R$-module and $q = \langle x_1,x_2,\ldots,x_s \rangle$ be a homogeneous $\mathfrak{m}$-primary ideal of $R$, then the Hilbert-Kunz multiplicity is defined by 
\[
e_{\mathrm{HK}}(q,M) = \lim_{n \to \infty}\frac{\ell_R(M/q^{[n]}M)}{n^d},
\]
where $q^{[n]}= \langle x_1^n, x_2^n, \ldots, x_s^n \rangle$. In general, it is not clear that this quantity is well defined. If $\mathrm{char}(k) = p > 0$, then for $n = p^e$, $\lim_{e \to \infty}\frac{\ell_R(M/q^{[p^e]}M)}{p^{ed}}$ always exists (see \cite{monsky}). If $q = \mathfrak{m}$, then we denote $e_{\mathrm{HK}}(\mathfrak{m},R)$ by $e_{\mathrm{HK}}(R)$. In this section, we give an explicit formula for the Hilbert-Kunz multiplicity of the co-ordinate ring of the projective monomial curve defined by $n_1,\ldots,n_p$, 
such that $n_1 < n_2 < \cdots < n_p$ and $\mathrm{gcd}(n_1,n_2, \ldots,n_p) = 1$.

\begin{theorem}\label{Hilbert-KunzMulti}
Let $S = \langle (0,n_p), (n_1,n_p-n_1), (n_2,n_p-n_2),\ldots, (n_{p-1},n_p-n_{p-1}), (n_p,0) \rangle$ be an affine semigroup in $\mathbb{N}^2$, where $n_1 < n_2 < \cdots < n_p$ and $\mathrm{gcd}(n_1,n_2, \ldots,n_p) = 1$. Put $n_0 = 0$. Then the Hilbert-Kunz multiplicity of $k[S]$ is 
\[
e_{\mathrm{HK}}\left( k[S] \right) = 1 + \frac{1}{n_p}\left( \sum_{r=1}^p (n_r-1)(n_r-n_{r-1}) \right) .
\]
\end{theorem}

\begin{proof}
Let $G$ be the group generated by $S$ in $\mathbb{Z}^2$. Then $G$ is a free $\mathbb{Z}$-module of rank 2. By \cite[Lemma 4.1]{cavaliereNiesi}, $G$ has a basis $\{(0,n_p), (1,-1)\}$. Let $\{(1,0), (0,1)\}$ be the canonical basis of $\mathbb{Z}^2$ as $\mathbb{Z}$-module. Then we have 
\[
(0,n_p) = 0 \cdot (1,0) + n_p (0,1)\quad \text{and} \quad (1,-1) = (1,0) - (0,1). 
\]
Therefore, the cardinality of $\frac{\mathbb{Z}^2}{G}$ is finite and equal to the modulus of the determinant of the matrix 
$\begin{bmatrix}
0 & n_p \\
1 & -1
\end{bmatrix} .
$
Therefore, we have $\vert \frac{\mathbb{Z}^2}{G} \vert = n_p$.
 
Let $J$ denote the ideal $\langle x^{n_p}, x^{n_1}y^{n_p-n_1}, \ldots , x^{n_{p-1}}y^{n_p - n_{p-1}}, y^{n_p} \rangle$ in $k[x,y]$. Observe that the radical ideal of $J$ is the maximal homogeneous ideal of $k[x,y]$. Therefore, the length of $\frac{k[x,y]}{J}$ is finite and equal to the $\mathrm{dim}_k \frac{k[x,y]}{J}$ as 
a $k$-vector space. Let $\mathcal{B}$ be the basis of $\frac{k[x,y]}{J}$ as $k$-vector space. Then by \cite[Theorem 39.6]{peeva}, observe that $\mathcal{B}= B \cup \bigcup_{r=1}^p B_r$, where $B = \{1,y,y^2, \cdots, y^{n_p-1}\}$ and for $r \in [1,p]$,
\[
B_r = \left\lbrace x^i y^j \mid 1 \leq i \leq n_r-1, \quad n_p-n_r \leq j \leq n_p - n_{r-1} -1   \right\rbrace.
\]

The cardinality of $B$ is $n_p$ and the cardinality of $B_r$ is $(n_r-1)(n_r-n_{r-1})$, 
for each $r \in [1,p]$. Therefore, the length of $\frac{k[x,y]}{J}$ is 
\[
\ell_{k[x,y]}\left(\frac{k[x,y]}{J} \right) = n_p + \left( \sum_{r=1}^p (n_r-1)(n_r-n_{r-1}) \right). 
\]
Now, the result follows from \cite[Corollary 2.3]{etohilbert-kunzmultiplicity}. \qedhere

\end{proof}

\begin{coro}\label{Hilbert-KunzArithmetic}
Let $S = \langle (0,n_p), (n_0,n_p-n_0), (n_1,n_p-n_1),\ldots, (n_{p-1},n_p-n_{p-1}), (n_p,0) \rangle$, where $n_0,n_1,\ldots,n_p$ is a minimal arithmetic sequence of positive integers, 
such that $n_0 < n_2 < \ldots < n_p $ and $\mathrm{gcd}(n_0,n_1,\ldots,n_p) = 1$. Then the Hilbert-Kunz multiplicity of $k[S]$ is 
\[
e_{\mathrm{HK}}\left( k[S] \right) = n_0 + \frac{p(p+1)d^2}{2n_p},
\]
where $d$ is the common difference.
\end{coro}
\begin{proof}
By Theorem \ref{Hilbert-KunzMulti}, we have 
\begin{align*}
e_{\mathrm{HK}}\left( k[S] \right) = 1 + \frac{1}{n_0+pd} \left(  \sum_{r=1}^p (n_r-1)(n_r-n_{r-1}) \right) + (n_0-1)n_0
&= \frac{n_0^2 + d(n_1+ n_2 + \cdots + n_p)}{n_0+pd} \\
&= \frac{2n_0^2 + 2n_0pd + p(p+1)d^2}{2(n_0+pd)}.
\end{align*}
\end{proof}

\begin{example}{\rm 
Let $A = k[x^3, x^2y, xy^2, y^3]$ be the co-ordinate ring of the twisted cubic curve in the projective space $\mathbb{P}^3$. The affine semigroup parametrizing this curve is $S = \langle (0,3),(1,2),(2,1),(3,0) \rangle$. Therefore, by Corollary \ref{Hilbert-KunzArithmetic}, $e_{\mathrm{HK}}(A)= 2$.
}
\end{example}

\begin{example}{\rm 
Let $S = \langle (0,19), (7,12), (10,9), (13,6), (16,3), (19,0) \rangle$. Therefore, in the notation of Corollary \ref{Hilbert-KunzArithmetic}, we have $n_0 = 7$, $p = 4$ and $d = 3$. Hence $e_{\mathrm{HK}}(k[S]) = 7 + \frac{20 \cdot 9}{2 \cdot 19} = \frac{223}{19} $.
}
\end{example}

\begin{example}{\rm 
Let $A = k[x^4, x^3y, xy^3,y^4]$. Then it will correspond to the semigroup ring of the affine semigroup $S = \langle (0,4),(1,3),(3,1),(4,0) \rangle$. Therefore, by Theorem \ref{Hilbert-KunzMulti}, $e_{\mathrm{HK}}(A)= 1 + \frac{1}{4}\left( 0 \cdot 1 + 2 \cdot 2 + 3 \cdot 1  \right) = \frac{11}{4}$.
}
\end{example}

{\it Acknowledgement.} Experiments with the software GAP\cite{GAP4} package\cite{gap-ns} have provided the initial observations. The authors would like to thank the anonymous referee for valuable comments and suggestions.

\end{document}